\documentclass[11pt]{article}

\usepackage[T1]{fontenc}
\usepackage{lmodern}
\usepackage[margin=1in]{geometry}
\usepackage{amsmath,amssymb,amsthm}
\usepackage{enumitem}
\usepackage[hidelinks]{hyperref}
\usepackage{cite}
\usepackage{xcolor}
\theoremstyle{plain}
\newtheorem{theorem}{Theorem}[section]
\newtheorem{lemma}[theorem]{Lemma}

\newtheorem{conjecture}[theorem]{Conjecture}
\newtheorem{claim}[theorem]{Claim}
\newtheorem{observation}[theorem]{Observation}

\theoremstyle{definition}

\newtheorem{openproblem}[theorem]{Open Problem}

\newcommand{\chrom}{\chi}

\newcommand{\kcol}{$k$-\textsc{Colouring}}

\title{Obstructions to $k$-colouring $H$-free graphs}

\author{Iain Beaton$^{a}$ \and Ben Cameron$^{b,*}$ \and Adam van Omme$^{b}$}
\date{\small
$^{a}$Department of Mathematics and Statistics, Acadia University, Wolfville, NS B4P 2R6, Canada\\
$^{b}$School of Mathematical and Computational Sciences, University of Prince Edward Island, 550 University Avenue, Charlottetown, PE C1A 4P3, Canada\\[4pt]
$^{*}$Corresponding author. e-mail: brcameron@upei.ca\\[4pt]
\today}

\begin{document}
\maketitle

\begin{abstract}
A graph is $H$-free if it has no induced subgraph isomorphic to $H$. In 2020, Chudnovsky, Goedgebeur, Schaudt, and Zhong characterized all graphs $H$ such that there are only finitely many minimal obstructions to $3$-colouring $H$-free graphs. 
In general, the minimal obstructions to $k$-colouring $H$-free graphs are the $(k+1)$-vertex-critical $H$-free graphs, those are, the $H$-free graphs $G$ with $\chi(G)=k+1$ but $\chi(G-v)=k$ for every vertex in $G$. 
In this paper we complete the characterization for all $k > 4$ by showing that there are onky finitely $k$-vertex-critical $H$-free graphs if and only if $H$ is an induced subgraph of $P_4+\ell P_1$ for some $\ell \geq 0$.
\end{abstract}

\section{Introduction}\label{sec:intro}
A graph $G$ with vertex set $V$ is $k$-colourable if there exists a mapping $\varphi:V \rightarrow \{1, \ldots, k\}$ such that $\varphi(u) \neq \varphi(v)$ whenever $u \sim v$ (i.e., whenever $uv$ is an edge). 
The \kcol{} problem asks whether or not an input graph $G$ is $k$-colourable. 
The \kcol{} problem for $k \geq 3$ has been an active area of research since 1972, when it was shown to be NP-complete as one of Karp's original 21 problems~\cite{Karp1972}.
The chromatic number of a graph $G$, denoted $\chi(G)$, is the smallest $k$ such that $G$ is $k$-colourable.
Since $k$-colourability is a hereditary property (i.e., it is closed under taking induced subgraphs), when $G$ is not $k$-colourable, there exists an induced subgraph $H$ with $\chi(H)=k+1$ that obstructs $G$ from being $k$-colourable.
The minimal obstructions that prevent a graph from being $k$-colourable are called $(k+1)$-vertex-critical graphs, where a graph is $(k+1)$-vertex-critical if $\chi(G)=k+1$ but every proper induced subgraph of $G$ is $k$-colourable.
Clearly, there are infinitely many minimal obstructions to $k$-colourability in general for all $k\ge 2$, and there is no known efficient recognition algorithm for these minimal obstructions for any $k\ge 3$.
However, there are graph classes that contain only finitely many $(k+1)$-vertex-critical graphs for some (or all) $k\ge 4$. 
If such a graph class is hereditary, then there is a polynomial-time algorithm that provides ``no''-certificates for solving \kcol{} when the inputs are restricted to graphs from this graph class  by simply iteratively searching the input graph for each of the $(k+1)$-vertex-critical graphs as an induced subgraph and returning one, if found, as a certificate
(see~\cite{P5banner2019}, for example, for full details).
McConnell, Mehlhorn, N\"{a}her, and Schweitzer~\cite{McConnell2011} make convincing arguments that certifying algorithms are highly desirable given the ease with which they can be robustly tested; however, the certificate must be easily (i.e., efficiently) verifiable to provide value.

It is well known that a class of graphs is hereditary if and only if it is defined by a set of forbidden induced subgraphs.
We say a graph is $H$-free if it does not contain $H$ as an induced subgraph, and, more generally, a graph is $(H_1,H_2,\dots, H_n)$-free if it does not contain $H_i$ as an induced subgraph for all $i\in \{1,2,\dots,n\}$.
Bruce, Ho\`{a}ng, and Sawada~\cite{Bruce2009} initiated the study of critical graphs in hereditary graph classes, and their work was later extended by Maffray and Morel~\cite{MaffrayMorel2012}, who showed that there are exactly twelve $4$-vertex-critical $P_5$-free graphs.
This complete list was used to develop a simple linear-time certifying algorithm to solve $3$-\textsc{Colouring} for $P_5$-free graphs.
The vertex-critical graphs in the class of $P_5$-free graphs are of particular interest as Ho\`{a}ng, Kami{\'{n}}ski, Lozin, Sawada, and Shu~\cite{Hoang2010} developed polynomial-time \kcol{} algorithms for all $k$ for $P_5$-free graphs, where a valid $k$-colouring is returned if one exists.
Unfortunately, Ho\`{a}ng, Moore, Recoskie, Sawada, and Vatshelle~\cite{Hoang2015} showed that there are infinitely many $k$-vertex-critical $H$-free graphs when $H$ contains an induced $2P_2$, for each $k\ge 5$.
Thus, for $k\ge 5$, considerable research interest has been focused on the $k$-vertex-critical graphs in subclasses of $P_5$-free graphs, and finiteness for all $k$ has been shown for many such subclasses, including $(P_5,H)$-free graphs when $H$ is any $(2P_2,K_3+P_1)$-free graph of order $4$~\cite{KCameron2021}, any $(2P_2,K_3+P_1)$-free graph of order $5$ except for $\overline{P_3+2P_1}$, $K_5-e$, and $K_5$ for which the problem remains open)~\cite{Brause2022,Kaminski2019,AbuadasCameronHoangSawada2024,Jooken2026,Dhaliwal2017,CaiGoedgebeurHuang2023,Xiaetal2023,xia2024results,P5W4Journal,BelavadiHoang2026,BeatonCameron2026P5Hbipartite}, and various other graphs $H$ of order larger than $5$~\cite{xia2024results,Kaminski2019,BeatonCameron2026P5Hbipartite}.
While this extensive list of positive results might initially give the impression that most subclasses of $P_5$-free graphs contain only finitely many $k$-vertex-critical graphs for all $k$, some very restrictive subclasses do indeed contain infinitely many. 
It was recently shown, using the constructions due to the second author and Ho\`{a}ng~\cite{CameronHoang2023P5C5}, that there are infinitely many $k$-vertex-critical $(2P_2,K_3+P_1,C_5,\overline{C_6},\operatorname{net},\operatorname{co\text{-}net})$-free graphs for each $k\ge 6$~\cite{BeatonCameron2026P5Hbipartite}.
Further, there are infinitely many $k$-vertex-critical $H$-free graphs when $H$ contains an induced claw (see, for example,~\cite{CameronHoangSawada2022}) or cycle (following from~\cite{Erdos}) for all $k\ge 3$.

The class of $P_5$-free graphs is not a maximal hereditary class in which \kcol{} is polynomial-time solvable for all $k$ (assuming P$\neq$NP), as Couturier, Golovach, Kratsch, and Paulusma~\cite{Couturier2015} showed that the same is true for $(P_5+\ell P_1)$-free graphs for all $\ell \ge 0$.
Their algorithms also return a valid $k$-colouring, if one exists, to serve as a ``yes'' certificate.
Further, for $k=3$, i.e., the smallest value for which \kcol{} is NP-complete in general, there are also polynomial-time algorithms that solve \kcol{} and return a valid $3$-colouring, if one exists, for $P_6$-free graphs~\cite{RanderathSchiermeyer2010} and $rP_3$-free graphs~\cite{HajebiLiSpirkl2022}.
These results were further bolstered by the most comprehensive result on vertex-critical graphs in hereditary classes: the following dichotomy theorem of Chudnovsky, Goedgebeur, Schaudt, and Zhong~\cite{Chud4critical2020}.

\begin{theorem}[\cite{Chud4critical2020}]\label{thm:4critdichot}
    There are finitely many $4$-vertex-critical $H$-free graphs if and only if $H$ is an induced subgraph of $P_6$, $2P_3$, or $P_4+\ell P_1$ for any $\ell\ge 0$. 
\end{theorem}

This result was itself a strengthening of the same authors' earlier result that, for connected $H$, there are only finitely many $4$-vertex-critical $H$-free graphs if and only if $H$ is a connected induced subgraph of $P_6$~\cite{Chudnovsky4criticalconnected2020}.
A corollary of this and the previously stated results for $H$-free graphs when $H$ contains an induced claw, cycle, or $2P_2$ is that the only graphs $H$ for which there may be only finitely many $k$-vertex-critical $H$-free graphs for any $k>4$ are induced subgraphs of $P_4+\ell P_1$ for some $\ell\ge 0$. 
This inspired interest in studying the classes defined by forbidding these induced subgraphs. 
K.~Cameron, Goedgebeur, Huang, and Y. Shi~\cite{KCameron2021} showed that there are only finitely many $k$-vertex-critical $(P_3+P_1)$-free graphs for all $k$.
The second author, Ho\`{a}ng, and Sawada~\cite{CameronHoangSawada2022} showed that there are only finitely many $k$-vertex-critical $(P_2+\ell P_1)$-free graphs for all $k$ and $\ell$, and posed the following conjecture.

\begin{conjecture}[\cite{CameronHoangSawada2022}]\label{conj:P4UellP1}
    Fix $k\ge 5$. There are only finitely many $k$-vertex-critical $H$-free graphs if and only if $H$ is an induced subgraph of $P_4+\ell P_1$ for some $\ell\ge 0$.
\end{conjecture}

To close the gap further on this conjecture, Abuadas, the second author, Ho\`{a}ng, and Sawada~\cite{AbuadasCameronHoangSawada2024} showed that Conjecture~\ref{conj:P4UellP1} is true for all $k\ge 5$ when $H$ is $P_3+\ell P_1$ for every $\ell\ge 0$ and for the $(P_4+P_1,\overline{P_4+P_1})$-free graphs~\cite{AbuadasCameronHoangSawada2024}.
Thus, to solve Conjecture~\ref{conj:P4UellP1}, it remained only to determine for which values $k>4$ and $\ell>0$ it is true for $H=P_4+\ell P_1$.
In particular, the case $\ell=1$ and $k=5$ remained open.
This prompted the first and second authors to show that there are only finitely many $k$-vertex-critical $(P_4+P_1, H)$-free graphs for all graphs $H$ of order $4$ and to ask for which graphs $H$ of order $5$ is the same true~\cite{BeatonCameron2025cogemfreeord4finite}.
Belavadi and Karthick very recently answered this question in the affirmative by showing that the result holds for $H\in\{\text{paraglider},\text{dart},\text{house}\}$~\cite{KarthickBelavadi2026}.
Beyond $\ell=1$, the first and second authors showed a variety of graphs $H$ for which there are only finitely many $k$-vertex-critical $(P_4+\ell P_1,H)$-free graphs for all $k$~\cite{BeatonCameron2026IWOCA,BeatonCameron2026P5Hbipartite}. 
In this paper, we prove Conjecture~\ref{conj:P4UellP1} for all $k$ and $\ell$ with the following dichotomy theorem. 

\begin{theorem}\label{thm:P4UellP1}
    Fix $k\ge 5$. There are only finitely many $k$-vertex-critical $H$-free graphs if and only if $H$ is an induced subgraph of $P_4+\ell P_1$ for some $\ell\ge 0$.
\end{theorem}

As noted above, this implies the existence of polynomial-time certifying algorithms to solve \kcol{} for the class of $(P_4+\ell P_1)$-free graphs for all $k\ge 1$ and $\ell\ge 0$.

The rest of this paper proceeds as follows.
Theorem~\ref{thm:P4UellP1} is proved in Section~\ref{sec:P4ellP1proof}.
We first provide some preliminary results, definitions, and notation in Section~\ref{sec:prelim}.
We conclude the paper in Section~\ref{sec:conc} with directions for future work.

\section{Preliminaries}\label{sec:prelim}
    For standard graph theory notation, we will follow~\cite{WestGraphTheoryBook}; however, we include some definitions and notations now to aid the reader.  
    For a graph $G$ and a subset of its vertices, $S$, we let $N(S)$ denote the open neighbourhood of $S$, $N[S]$ the closed neighbourhood, and $G[S]$ the subgraph induced by $S$.
    We call the set $V(G)\setminus N[S]$ the \textit{non-neighbourhood} of $S$.
    Further, we say that a subset $X$ of the vertices of a graph $G$ is \textit{complete} (respectively, \textit{anticomplete}) to $S$ if $x\sim s$ (respectively, $x\nsim s$) for all $x\in X$ and $s\in S$.
    If $X$ is complete (or anticomplete) to $S$ and $x\in X$, we also say that the vertex $x$ is complete (or anticomplete) to $S$.
    The set $S$ is a \emph{module} if every vertex of $V(G)\setminus S$ is either complete or anticomplete to $S$.
    For graphs $G$ and $H$, we let $G+H$ denote the disjoint union of $G$ and $H$.
    For a non-negative integer $m$, we let $mG$ denote the disjoint union of $m$ copies of $G$.

    For positive integers $\ell,c,t$, write $R^{(\ell)}_c(t)$ for the least integer
    $N$ such that every colouring of the $\ell$-element subsets of an $N$-element
    set with at most $c$ colours admits a $t$-element subset all of whose
    $\ell$-element subsets receive the same colour. The
    existence and finiteness of $R^{(\ell)}_c(t)$ is Theorem $B$ from Ramsey's seminal paper~\cite{Ramsey}.

    The proof of Theorem \ref{thm:P4UellP1} heavily relies on the structure of co-trees of co-graphs (i.e. $P_4$-free graphs). The co-tree of a co-graph is formed as follows.
    The leaves of the co-tree are the vertices of the co-graph. Choose a clique (or independent set) $S$ of maximal size such that $S$ is a module. 
    Collapse $S$ into a single vertex and label it 1 when $S$ is a clique and $0$ when $S$ is an independent set. 
    Add this new vertex to the co-tree and join it to each vertex in the co-tree corresponding to the vertices in $S$. 
    Repeat this process on the new graph until you get to a single vertex. 
    Note depending on your choice of modules, many co-trees can be obtained from the same co-graph.
    However, there is always a unique co-tree such that every root to leaf path alternates between 0 and 1 vertices \cite{CORNEIL1981163}. 
    We call such a co-tree \emph{reduced}.
    We will exploit the following facts about co-trees

\begin{observation}
    Let $G$ be a $P_4$-free graph and $\mathcal{T}$ be the reduced co-tree of $G$. Then

    \begin{itemize}
        \item Every vertex in a co-tree corresponds to a module in $G$.
        \item Every non-leaf vertex in a co-tree has at least two children.
        \item Every root to leaf path in a reduced co-tree alternates between 0 and 1 vertices. (See \cite{CORNEIL1981163})
    \end{itemize}
\end{observation}

    The following two lemmas will be useful in the proof of our main result.

    \begin{lemma}[\cite{P5W4Journal}]\label{lem:edgemixhomo}
        For disjoint sets $H, A \subseteq V(G)$, if $v \in A$ is not mixed on any edge in $H$, then $v$ is not mixed on any connected component of $H$.   
    \end{lemma}

    \begin{lemma}[\cite{KCameron2021}] \label{lem:XY}
		Let $G$ be a $k$-vertex-critical graph. There does not exist two vertex subsets $X, Y \subseteq V(G)$ satisfying all of the following conditions.
		\begin{itemize}
			\item $X$ and $Y$ are anticomplete to each other.
			\item $\chi(G[X])\le\chi(G[Y])$.
			\item Y is complete to $N(X)$.
		\end{itemize}
	\end{lemma}

\section{$k$-vertex-critical $(P_4+\ell P_1)$-free graphs}
\label{sec:P4ellP1proof}

In this section, we will prove our main theorem.

\begin{proof}[Proof of Theorem~\ref{thm:P4UellP1}] Let $G$ be a $k$-vertex-critical $(P_4+\ell P_1)$-free graph.
We will bound $\alpha(G)$ and hence bound $|V(G)|$ through Ramsey's theorem~\cite{Ramsey} since $G$ is necessarily $K_{k+1}$-free since it is $k$-vertex-critical.
Note that it is sufficient to bound the size of the non-neighbourhood of any independent set of size $\ell$. 
Fix an independent set $S=\{s_1,\dots,s_\ell\}$ (noting that if $\alpha(G)<\ell$, then we are done) and let
\[
U=V(G)\setminus N[S],\qquad X=N(S).
\]
Note that $G[U]$ is $P_4$-free and hence a co-graph.
Thus, it suffices to show the reduced co-tree of $G[U]$ has bounded order.
Moreover it suffices to bound the depth and number of children of each non-leaf vertex in the reduced co-tree of $G[U]$.
Call this reduced co-tree $\mathcal{T}$.

We first bound the height of $\mathcal{T}$ by $2(k-1)-1$.
Suppose the height is at least $2(k-1)$.
Choose a vertex $v$ of maximum depth in $\mathcal T$.
Note that $v$ is necessarily a leaf and therefore also vertex of $G[U]$.
Further, let $P_v$ denote the unique path from $v$ to the root of $\mathcal T$, which necessarily alternates vertices labeled 0 (disjoint union vertices) and 1 (join vertices), although not necessarily starting with 0.
So there are at least $k-1$ join vertices on $P_v$. Let $\{j_1,j_2,\dots j_{k-1}\}$ be a set of join vertices on $P_v$.
At each $j_i$, traversing a path from $j_i$ through a child that is not on $P_v$ until we arrive at a leaf vertex, which we call $k_i$.
Note that $v\neq k_i$ and $k_i\neq k_{i'}$ since the path from each $j_i$ to $k_i$ intersects $P_v$ only in the vertex $j_i$.
Now, let $K = \{v, k_1,k_2,k_3,\dots,k_{k-1}\}$.   
Since each vertex in $K$ was chosen as a descendant of a join vertex, $K$ induces $K_k$ in $G[U]$, a proper subgraph of $G$, which contradicts the fact that $G$ is $k$-vertex-critical.

We now bound the number children of each non-leaf vertex in the reduced co-tree. 
Note each non-leaf vertex corresponds to a module of $G[U]$ which is either a join or a disjoint union of smaller modules in $G[U]$.
Thus, if a join vertex has more than $k-1$ children, $G[U]$ and hence $G$ would contain an induced $K_k$ contradicting the fact that $G$ is $k$-vertex-critical.
Thus, it suffices to bound the number of children in a disjoint union vertex.
More specifically, we will bound the number of components in any module in $G[U]$.
The remainder of the proof is a series of claims dedicated to proving this outstanding assertion in Claim~\ref{cla:boundmodulecomponents}.

Fix a module $W$ of $U$ such that $G[W]$ is a disconnected graph with components $C_1,\dots,C_m$. 
For each $x\in X$,
define its \emph{support}, denoted $J_x$, as
\[
J_x=\{i\in[m]:x\text{ has a non-neighbour in }C_i\}.
\]
Thus $i\notin J_x$ means that $x$ is complete to $C_i$. A support is called \emph{large} if its size exceeds $\ell$.

\begin{claim}
\label{cla:uniform}
If $x\in X$ is mixed on a component $C_i$, then $|J_x|\le\ell$. Consequently, a
vertex with large support is either complete or anticomplete to each
component, and is anticomplete to precisely the components indexed by its
support.
\end{claim}

\begin{proof}[Proof of Claim~\ref{cla:uniform}]
Let $x\in X$ such that $x$ is mixed on a component $C_i$.
Choose $s_a\in S$ adjacent to $x$.
Since $C_i$ is connected, there are adjacent vertices $u,w\in C_i$ with $xu\in E(G)$
and $xw\notin E(G)$ from Lemma~\ref{lem:edgemixhomo}. The vertices $s_a,x,u,w$, in that order, induce $P_4$. If $J_x$ contained $\ell$ indices other than $i$, then let $j_1,j_2,\dots j_{\ell}\in J_{x}\setminus\{i\}$.
Let $x_{h}$ be a vertex in $C_{j_h}$ for each $h\in \{1,2,\ldots,\ell\}$. Now, $\{s_a,x,u,w,x_1,x_2,\dots, x_{\ell}\}$ induces $P_4+\ell P_1$, a contradiction. Therefore, $|J_x|\le\ell$.
\end{proof}

Partition $X$ into $X_1,\dots,X_\ell$ by assigning each $x\in X$ to one
adjacent $s_a$ (breaking ties arbitrarily if $x$ is adjacent to more than
one vertex of $S$). Thus, $X_a\subseteq N(s_a)$.
For each $a\in \{1,2,\dots,\ell\}$, consider the family $\mathcal J_a$ of distinct large supports
$J_x$ with $x\in X_a$, that is,  $\mathcal{J}_a=\{J_x:|J_x|>\ell \text{ and } x\in X_a\}$. 
Recall that every $x\in X_a$ with a large support $J_x$ is necessarily complete or anticomplete to every component $C_1,C_2,\dots C_m$ by Claim~\ref{cla:uniform}.

\begin{claim}\label{cla:intersection}
    For each $a\in \{1,2,\dots \ell\}$, two incomparable members of $\mathcal J_a$ intersect in at most $\ell-1$ indices.
\end{claim}
\begin{proof}[Proof of Claim~\ref{cla:intersection}]
    Suppose $J_x,J_y\in\mathcal J_a$ are incomparable (i.e, neither is a subset of the other) and their intersection has at least $\ell$ indices.
    Since they are incomparable, choose $i\in J_x\setminus J_y$ and
    $j\in J_y\setminus J_x$. Choose
    $u\in C_i$, $w\in C_j$, and one vertex from
    each of $\ell$ components indexed by a fixed $\ell$-subset of $J_x\cap J_y$, call these $\ell$ vertices $Z$.
    Note by definition of $\mathcal{J}_a$ we have $s_a \sim x$ and $s_a \sim y$.
    Now, $Z$ is an independent set since each vertex is from a distinct component of $G[W]$ and anticomplete to $u,w,s_a, x, y$.
    If $x\sim y$, then $\{w,x,y,u\}\cup Z$ induces $P_4+\ell P_1$, a contradiction, and if $x\nsim y$, then $\{w,x,s_a,y\}\cup Z$ induces $P_4+\ell P_1$, a contradiction.
\end{proof}

\begin{claim}\label{cla:largesupportsinduceedge}
    Let $x,y\in X_a$ and  $J_x,J_y\in \mathcal J_a$ with $x\neq y$ for some $a\in \{1,2,\dots,\ell\}$. If $J_x\subset J_y$, then $xy\in E(G)$.
\end{claim}
\begin{proof}[Proof of Claim~\ref{cla:largesupportsinduceedge}]
    Suppose $J_x\subset J_y$ are large supports of $x,y\in X_a$ and suppose by way of contradiction that $xy$ is not an edge.
    Since the containment is strict, choose $i\in J_y\setminus J_x$, and
    $u\in C_i$. 
    Since $i\in J_y$, $y$ is anticomplete to $C_i$ by
    Claim~\ref{cla:uniform}, so $yu\notin E(G)$. Since $i\notin J_x$ then $x$ is
    complete to $C_i$, so $xu\in E(G)$. Thus, $\{u,x,s_a,y\}$ induces a $P_4$ (in that order). 
    For a fixed $\ell$-subset $J'$ of $J_x$, choose one vertex from each component indexed by $J'$ and call these vertices $Z'$. We now have $\{u,x,s_a,y\}\cup Z'$ inducing a $P_4+\ell P_1$, a contradiction. Hence $xy\in E(G)$.
\end{proof}

\begin{claim}\label{cla:boundonlargestchaininJa}
    For each $a\in\{1,2,\dots,\ell\}$, every strict inclusion chain in $\mathcal J_a$ has at most $k-1$ members. Therefore $\mathcal J_a$ can be partitioned into the union of at most $k-1$ antichains. 
\end{claim}
\begin{proof}[Proof of Claim~\ref{cla:boundonlargestchaininJa}]
    Take a strict inclusion chain
    $J_{x_1}\subset J_{x_2}\subset\cdots\subset J_{x_p}$ in
    $\mathcal J_a$ for some $a\in\{1,2,\dots,\ell\}$ with $x_1,\dots,x_p\in X_a$ representing the chain's
    members. By Claim~\ref{cla:largesupportsinduceedge}., $x_i\sim x_{j}$ for every pair $1\le i<j\le p$.
    Thus, $\{x_1,\dots,x_p\}\cup\{s_a\}$ is a clique of size $p+1$. 
    If $p \ge k$, then $G$ contains a clique of size $k+1$, contradicting $G$ being $k$-colourable.
  
    Hence every strict inclusion chain in $\mathcal J_a$ has at most $k-1$ members.

     Now by Mirsky's Theorem~\cite{mirsky1971dual}, the minimum number of antichains to partition $\mathcal J_a$ is the length of the longest chain. 
     Thus $\mathcal J_a$ can be partitioned into at most $k-1$ antichains.  
 
\end{proof}

Now let $R=R^{(\ell+1)}_{\,2^{\ell (k-1)}}((k-1)\ell+2)$ be the Ramsey number as described in Section~\ref{sec:prelim}.
Note for any set $A$ with $|A| \geq R$ elements and any $2^{\ell (k-1)}$-colouring of the $(\ell+1)$-subsets of $A$ we have a subset $I \subseteq A$ with $|I| \geq (k-1)\ell+2$ such that every $(\ell+1)$-subset of $I$ was receives the same colour.

\begin{claim}\label{cla:ramsey}
Suppose $A\subseteq[m]$ has size at least $R$. There is a subset
$I\subseteq A$ of size $(k-1)\ell+2$ such that, for every $x\in X$, either
\[
I\subseteq J_x\quad\text{or}\quad|I\cap J_x|\le\ell.
\]
In the first case, $x$ is anticomplete to every component indexed by $I$.
\end{claim}

\begin{proof}[Proof of Claim~\ref{cla:ramsey}]
For every $x \in X$, if $|J_x| \leq \ell$ then $|I \cap J_x| \leq \ell$ regardless of the choice of $I$.
Moreover $I \not\subseteq J_x$ as $|I|=(k-1)\ell+2> \ell$.
Thus it suffices to only consider $x \in X$ such that $|J_x| > \ell$.

For each $a\in\{1,\dots,\ell\}$, partition $\mathcal{J}_a=\{J_x:|J_x|>\ell \text{ and } x\in X_a\}$ into the at most $k-1$ antichains given by Claim~\ref{cla:boundonlargestchaininJa}. There are at most
$\ell (k-1)$ such families in total, ranging over all $a$.
Label these families $\mathcal{F}_1, \ldots,  \mathcal{F}_{\ell (k-1)}$.

For every $T\subseteq A$ with $|T|=\ell+1$, assign $T$ a bit vector in $\{0,1\}^{\ell (k-1)}$ where $i^{th}$ coordinate of this bit vector is $1$ if and only if for some $J_x \in \mathcal{F}_i$ we have $T \subseteq J_x$.
Note there are at most $2^{\ell (k-1)}$ of these bit vectors in total. 
Since $|A|\ge R$ then by~\cite[Theorem~B]{Ramsey}, there exists a set $I\subseteq A$ of size $(k-1)\ell+2$ such that every $(\ell+1)$-subset of $I$ receives the same bit vector from $\{0,1\}^{\ell (k-1)}$. 

Now, for any $i\in \{1,2,\ldots,\ell (k-1)\}$, every $(\ell+1)$-subset of $I$ will have the same $i$th coordinate of their bit vector. 
Call the $i^{th}$ coordinate $\epsilon_i$.

Suppose $\epsilon_i=0$. 
By the definition of our bit vector, no $(\ell+1)$-subset of $I$ is contained in $J_x$ for any $J_x\in \mathcal{F}_i$.
Hence $I \not\subseteq J_x$.
Further, $J_x$ does not contain $\ell+1$ or more indices from $I$, so $|J_x \cap I| \leq \ell$.

So, suppose $\epsilon_i=1$, then every $(\ell+1)$-subset $T\subseteq I$ is contained in some
member of $\mathcal{F}_i$.
Let $I_{\ell+1}$ denote the set of all $(\ell+1)$-subsets of $I$ and define $\pi:I_{\ell+1}\to \mathcal F_i$ defined by $\pi(T)$ being a set $J_x\in \mathcal F_i$ with $T\subseteq J_x$. 
By Claim~\ref{cla:intersection} we have that $\pi$ is a well-defined function as distinct members of an antichain intersect in at most $\ell-1$ elements and hence $T$ is contained in a unique $\pi(T) \in \mathcal{F}_i$.
Now, for any $T,T'\in I_{\ell+1}$ differing in exactly one element, have $\pi(T)=\pi(T')$ by Claim~\ref{cla:intersection}.
Since any two $(\ell+1)$-subsets of $I$ can be connected by a sequence of successive one-element replacements, we have that all $(\ell+1)$-subset of $I$ map to the same set in $\mathcal{F}_i$ under $\pi$.
Let this set be $J^*$.
Note that $I\subseteq J^*$ since every element of $I$ is contained some member of $I_{\ell+1}$ and $T\subseteq J^*$ for every $T\in I_{\ell+1}$.
Thus, every other member $J \in \mathcal{F}_i$ has $|J \cap I| \leq |J\cap J^*|\le\ell$. 

Thus for every family $\mathcal F_i$ and $J_x \in \mathcal F_i$ we have that either $I \subseteq J_x$ or $|I \cap J_x| \leq \ell$.
Note for every $x \in X$ with $|J_x|> \ell$ there is a family $\mathcal F_i$ with $J_x \in \mathcal F_i$ and hence $I \subseteq J_x$ or $|I \cap J_x| \leq \ell$. 
Moreover if $|J_x| \leq \ell$ then $|I \cap J_x| \leq \ell$.
Thus for any $x \in X$ we have that either $I \subseteq J_x$ or $|I \cap J_x| \leq \ell$. 

Note when any $J_x$ satisfies $I \subseteq J_x$ we have that $|J_x| \geq |I| \geq \ell+2$.
Thus by Claim~\ref{cla:uniform}, we have that $x$ is complete or anticomplete to every component of $W$.
Since $J_x \neq \emptyset$, $x$ is anticomplete to each component indexed by $J_x$ and hence anticomplete to each component indexed by $I$.
\end{proof}

\begin{claim}\label{cla:boundmodulecomponents}
If $W$ is a module in $G[U]$ and $G[W]$ is disconnected, then $G[W]$ has
fewer than $M=(k-1)R$ connected components.
\end{claim}

\begin{proof}[Proof of Claim~\ref{cla:boundmodulecomponents}]
To show a contradiction suppose that $G[W]$ has at least $(k-1)R$ connected
components, with vertex sets $C_1,\dots,C_m$, for some $m\ge (k-1)R$.

Each component $C_i$ satisfies $\chrom(G[C_i])\in\{1,\dots,k-1\}$, thus by the pigeonhole principle, since $m\ge (k-1)R$, we must have at least $R$ of the components with the same chromatic number.
Let $r\in\{1,\dots,k-1\}$ and $A\subseteq[m]$ be such that $|A|\ge R$ and $\chrom(G[C_i])=r$ for each $i \in A$.
By Claim~\ref{cla:ramsey} there exists a subset $I\subseteq A$ of size $(k-1)\ell+2$ such that every $x\in X$ either has $I\subseteq J_x$ 
or $|I\cap J_x|\le\ell$. 
Now, fix $i_0\in I$ and delete the whole component $C_{i_0}$.
Since $G$ is $k$-vertex-critical, then $G-C_{i_0}$ has a proper
$(k-1)$-colouring, say $\varphi$.
We call a colour $c\in\{1,\dots,k-1\}$ \emph{exceptional} if there exists an $x\in X$ that received colour $c$ under $\varphi$ and is such that $|J_x\cap I|\le\ell$. 
That is, $c$ is exceptional if $\varphi(x)=c$ and $x$ is complete to at least $|I|-\ell=(k-2)\ell+2$ components indexed by $I$.
If there are no exceptional colours, then for each $x\in X$, $|J_x\cap I|>\ell$ and therefore $|J_x|>\ell$. 
Thus, by Claim~\ref{cla:uniform}, each $x\in X$ is either complete or anticomplete to each component of to $C_i$ for each $i\in I$.
Further, from our choice of $I$, we must have $I\subseteq J_x$ for all $x\in X$, and therefore $X$ is anticomplete to $C_i$ for each $i\in I$.
Since $W$ is a module in $G[U]$ and $C_i\subseteq W$ is also anticomplete to $S\cup X$, it follows that $N(C_i)=N(C_j)$ for all $i,j\in I$ with $i\neq j$.
By Lemma~\ref{lem:XY}, this contradicts $G$ being $k$-vertex-critical.

Therefore, we may suppose that there is at least one exceptional colour.
For each exceptional colour $c$, choose one vertex $x_c\in X$ with $\varphi(x_c)=c$ and let
\[
D=\bigcup_{c\text{ exceptional}}(J_{x_c}\cap I).
\]

Recall $J_{x_c}$ is the set of indices of components not complete to $x_c$.
Thus any $j\in I\setminus(D\cup\{i_0\})$ is such that $C_j$ is complete to each $x_c$ and hence can not have any vertex coloured $c$.
Now there are at most $k-1$ exceptional colours (one per colour class).
So, $|D|\le (k-1)\ell$ since each contributes at most $\ell$ indices to $D$.
Moreover, $|I|=(k-1)\ell+2$ so there is an index $j\in I\setminus(D\cup\{i_0\})$ such that no exceptional colour occurs on any vertex of $C_j$ under $\varphi$.
Note that the restriction of $\varphi$ to $C_j$ is a proper colouring of $G[C_j]$ using non-exceptional colours.
Since $\chrom(G[C_j])=r$, then $\varphi$ restricted to $C_j$ uses at least $r$ distinct colours.
Choose $r$ of these colours that appear on $C_j$, forming a palette $P$
with $|P|=r$ with no member of $P$ being exceptional.

We now claim that no vertex in $N(C_{i_0})$ uses a colour of $P$ under $\varphi$.
Note that $N(C_{i_0}) \subseteq X \cup (U\setminus W)$, so we consider two cases.

\begin{itemize}
\item[Case 1.] Suppose $x\in X$ with $\varphi(x)\in P$. 
Since no colour of $P$ is exceptional, $\varphi(x)$ is not exceptional and hence $|J_x\cap I|>\ell$. 
By Claim~\ref{cla:ramsey}, we must then have that $I\subseteq J_x$ and $x$ is then anticomplete to every component indexed by $I$, in particular to $C_{i_0}$.
So, $x$ is not a neighbour of any vertex of $C_{i_0}$ as desired.

\item[Case 2.] Suppose $z\in U\setminus W$ with $\varphi(z) \in P$ and suppose by way of contradiction that $z$ has a neighbour in $C_{i_0}$. Then $z$ is not anticomplete to $W$. 
Since $W$ is a module in $G[U]$ then $z$ is complete to $W$.
Thus, $z$ is complete to $C_j\subseteq W$ and $\varphi(z) \notin P$, a contradiction.
\end{itemize}

Since $\chrom(G[C_{i_0}])=r$ and $|P|=r$, there is a proper colouring of $G[C_{i_0}]$ using exactly the colours of $P$. 
Combine this with the $(k-1)$-colouring $\varphi$ on $V(G)\setminus C_{i_0}$.
This proper colouring of $G$ uses only $k-1$ colours, contradicting $\chrom(G)=k$.
\end{proof}

Claim~\ref{cla:boundmodulecomponents} now bounds the number of components in any module of $G[U]$, and therefore the number of children any vertex with label 0 can have in $\mathcal T$. 
Thus, $\mathcal T$ has order bounded by a constant depending only on $k$ and $\ell$.
Since every vertex of $U$ is a leaf of $\mathcal T$, it follows that $|U|$ is also bounded by a constant depending only on $k$ and $\ell$.
Since $U$ was defined as the non-neighbourhood of an arbitrary independent set of size $\ell$ in $G$, it follows that the independence number is bounded by $\ell +|U|$.
Applying Ramsey's Theorem now completes the proof. 

\end{proof}

\section{Conclusion}\label{sec:conc}
In this paper, we completely determined for which graphs $H$ there are only finitely many $k$-vertex-critical $H$-free graphs for all $k\ge 5$. 
As $H$-free graph classes are among the most restrictive ways to define hereditary graph classes, the following open problem naturally generalizes the efforts that culminated in our dichotomy theorem.

\begin{openproblem}\label{prob:hereditary}
    For which values of $k>3$ and which hereditary classes of graphs $\mathcal H$ are there only finitely many $k$-vertex-critical graphs in $\mathcal H$?
\end{openproblem}

Many known results toward Open Problem~\ref{prob:hereditary} for $k>4$ were discussed in Section~\ref{sec:intro} for $(P_5,H)$-free graphs, but Theorem~\ref{thm:4critdichot} also strongly motivates considering subclasses of $P_6$-free graphs and of $2P_3$-free graphs.
In fact, showing that there are only finitely many $(k+1)$-vertex-critical graphs in a subclass of $P_6$-free graphs for any $k>4$ could lead to new polynomial-time \kcol{} algorithms, as \kcol{} for $P_6$-free graphs remains NP-complete for all $k>4$~\cite{Huang2016}.
Aside from the noted subclasses of $P_5$-free graphs, it is known that there are only finitely many $5$-vertex-critical $(P_6,H)$-free graphs when $H=$ banner~\cite{Huang2019} or $H=$ bull~\cite{5critP6bullfree}.
On the other hand, the class of $rP_3$-free graphs was recently shown to admit polynomial-time \kcol{} algorithms for all $k$ and $r$~\cite{ChudnovskyHajebiSpirkl2024}. 
These algorithms, together with the fact that there are infinitely many $k$-vertex-critical $2P_2$-free graphs for all $k\ge 5$, strongly motivate studying $k$-vertex-critical graphs in subclasses of $2P_3$-free graphs for $k>4$.
Other than the subclasses of $2P_2$-free graphs considered in~\cite{Adekanye2024IWOCA}, vertex-critical graphs have not been studied in any other subclasses of $2P_3$-free graphs.
Thus, as a subproblem of Open Problem~\ref{prob:hereditary}, we pose the following.

\begin{openproblem}
    For which graphs $H$ are there only finitely many $k$-vertex-critical $(2P_3,H)$-free graphs for all $k\ge 1$?
\end{openproblem}

Since the dichotomy of Theorem~\ref{thm:4critdichot} applies only to $H$-free graphs, Open Problem~\ref{prob:hereditary} is also open for $k=4$.
Hereditary classes for which it is known that there are only finitely many $4$-vertex-critical graphs include $(P_7,K_3)$-free graphs~\cite{criticalP7K3free}, which was very recently extended by showing that there are exactly three such graphs~\cite{zhouetal2025P11C3freecritical}.
Also in~\cite{zhouetal2025P11C3freecritical}, an infinite class of $4$-vertex-critical $(4K_2,K_3)$-free graphs was constructed, and the problem of finding the smallest $t$ such that there are only finitely many $4$-vertex-critical $(P_t,K_3)$-free graphs was posed.
Thus, the answer lies somewhere between $t=7$ and $t=10$.
For $t=8$, it is known that $(P_8,K_3,C_4)$-free graphs and $(P_8,K_3,C_5)$-free graphs are $3$-colourable~\cite{ChudnovskyStacho2018}.

The question for $4$-vertex-critical $(P_t,K_3)$-free graphs is a subproblem of determining the number of $k$-vertex-critical $(H,K_n)$-free graphs for various graphs $H$ and $n<k$. 
One of only three open cases of an open problem posed in~\cite{KCameron2021} to determine exactly which graphs $H$ of order $5$ admit only finitely many $k$-vertex-critical $(P_5, H)$-free graphs for any $k\ge 5$ is $H=K_5$ (see~\cite{BeatonCameron2026P5Hbipartite} for the reduction to these three cases).
In fact, it follows directly from~\cite{Hoang2015} that there are infinitely many $5$-vertex-critical $(P_5,K_5)$-free graphs. 
This is especially interesting given that, for $k\ge 16$, there are no $k$-vertex-critical $(P_5,K_5)$-free graphs since all such graphs are $15$-colourable~\cite{Esperetetal2013}.
The somewhat surprising behaviour of the number of $k$-vertex-critical $(P_5,K_5)$-free graphs as $k$ increases is easily explained by the fact that $\chi(K_5)=5$, and therefore forbidding $K_5$ removes at most one of the infinitely many $5$-vertex-critical $P_5$-free graphs. 
The fact that there are no $16$-chromatic $(P_5,K_5)$-free graphs makes this case somewhat of a technicality.
This prompts us to pose the following open problem. 

\begin{openproblem}\label{prob:infinitethenfinite}
    Fix $k\ge 4$. Is there a finite list of nonisomorphic graphs $H_1,H_2,\dots H_n$ with $\chi(H_i)<k$ for all $i\in\{1,2,\dots,n\}$ such that the class of $(H_1,H_2,\dots,H_n)$-free graphs contains, for some $m\ge k$, infinitely many $m$-vertex-critical graphs and $(m+1)$-chromatic graphs, but only finitely many $(m+1)$-vertex-critical graphs?
\end{openproblem}

Clearly, the answer to Open Problem~\ref{prob:infinitethenfinite} is ``no'' if $H_i$ has no universal vertices for each $i\in\{1,\dots,n\}$ since joining a vertex to each graph in an infinite class of $m$-vertex-critical $(H_1,H_2,\dots,H_n)$-free graphs will produce an infinite class of $(m+1)$-vertex-critical $(H_1,H_2,\dots,H_n)$-free graphs.

\section*{Acknowledgments}
The research of Iain Beaton was supported by the Natural Sciences and Engineering Research Council of Canada (NSERC) grants RGPIN-2025-06012 and DGECR-2025-00001.
The research of Ben Cameron was also supported by NSERC, grants RGPIN-2022-03697 and DGECR-2022-00446. 

\section*{Declaration of AI-Assisted Content}

ChatGPT 6 Astra (referred to as the model throughout this section) was used in the development of the results of this paper. 
On September 1, 2026, the second and third author spent a few hours prompting the model to construct new infinite classes of $k$-vertex-critical $P_5$-free graphs.
We provided the following papers~\cite{CameronHoang2023P5C5,BeatonCameron2025cogemfreeord4finite,BeatonCameron2026IWOCA,BeatonCameron2026P5Hbipartite} for constructions and structure of the infinite classes that are already known.
After much back-and-forth, it did provide one for each $k\ge 7$, albeit without new forbidden induced subgraphs. 
This encouraged us to the model's potential and the third author then began a multi-day session with the model with to attack the open problems in this area for both sides, i.e., attempt to construct infinite classes of $k$-vertex-critical graphs with new forbidden attempt to prove new finiteness results for the vertex-critical $P_5$-free and $(P_4+\ell P_1)$-free graphs. 
For finiteness, the third author directed the model to simultaneously consider building on the results the first and second author recently obtained in~\cite{BeatonCameron2026P5Hbipartite} by starting with a large induced half-graph or crown-graph, and to try ``new'' ideas.

On September 9, 2026, the model returned a very elegant and simple proof that there are only finitely many $k$-vertex-critical  $(P_4+\ell P_1,K_3)$-free graphs for all $k$, solving a persistent open problem noted by the first and second authors in~\cite{BeatonCameron2026IWOCA}.
With this, the second and third author began another concerted effort with the model to solve the bigger open problem 
(i.e., Theorem~\ref{thm:P4UellP1}) graphs with an hour-long session to try and extend the triangle-free proof to general $K_n$-free graphs, and to guide it toward what we determined to be the most promising partial results it produced to that point.
The third author then continued another multi-day session with the model and the their independently developed computational models. 
On September 14, 2026, this produced a 27-page proof of Theorem~\ref{thm:P4UellP1}. 
This proof was then completely rewritten by the first and second author to the 4-page version included in this paper. 
The authors did fix some small holes in the arguments generated by the model, but the key ideas in the proof (i.e., bounding the order of the co-tree in the non-neighbourhood of an arbitrary $\ell$-independent set and the ideas to bound the number of children of the 0-nodes) are those generated by the model.

The introduction, conclusion, and this declaration were developed and written by the first and second authors, with the model being used only for a last pass in the review process and only to find uncaught grammatical and typographical errors.
The authors take full responsibility for the contents of this paper and its correctness.

\bibliographystyle{abbrv}
\bibliography{refs}

\begin{thebibliography}{10}

\bibitem{AbuadasCameronHoangSawada2024}
T.~Abuadas, B.~Cameron, C.~T. Hoàng, and J.~Sawada.
\newblock Vertex-critical $({P}_3+\ell {P}_1)$-free and vertex-critical (gem,
  co-gem)-free graphs.
\newblock {\em Discrete Appl. Math.}, 344:179--187, 2024.

\bibitem{Adekanye2024IWOCA}
M.~Adekanye, C.~Bury, B.~Cameron, and T.~Knodel.
\newblock {On the Finiteness of k-Vertex-Critical $2P_2$-Free Graphs with
  Forbidden Induced Squids or Bulls}.
\newblock In A.~A. Rescigno and U.~Vaccaro, editors, {\em Combinatorial
  Algorithms}, pages 301--313, Cham, 2024. Springer Nature Switzerland.

\bibitem{BeatonCameron2025cogemfreeord4finite}
I.~Beaton and B.~Cameron.
\newblock Vertex-critical graphs in co-gem-free graphs.
\newblock {\em Theoretical Computer Science}, 1042:115234, 2025.

\bibitem{BeatonCameron2026P5Hbipartite}
I.~Beaton and B.~Cameron.
\newblock {A dichotomy for the number of vertex-critical ($P_5$, $H$)-free
  graphs when $H$ is bipartite}.
\newblock {\em preprint, arXiv:2608.20045 [math.CO]}, 2026.

\bibitem{BeatonCameron2026IWOCA}
I.~Beaton and B.~Cameron.
\newblock Vertex-critical graphs in subfamilies of $({P}_4+\ell {P}_1)$-free
  graphs.
\newblock In F.~Foucaud and A.~Parreau, editors, {\em Combinatorial
  Algorithms}, pages 118--130, Cham, 2026. Springer Nature Switzerland.

\bibitem{BelavadiHoang2026}
M.~Belavadi and C.~T. Hoàng.
\newblock {Structural description of (bull, house)-free graphs}.
\newblock {\em preprint, arXiv:2604.27594 [math.CO]}, 2026.

\bibitem{KarthickBelavadi2026}
M.~Belavadi and T.~Karthick.
\newblock {Vertex-critical co-gem-free graphs}.
\newblock {\em preprint, arXiv:2606.11757 [math.CO]}, 2026.

\bibitem{criticalP7K3free}
F.~Bonomo-Braberman, M.~Chudnovsky, J.~Goedgebeur, P.~Maceli, O.~Schaudt,
  M.~Stein, and M.~Zhong.
\newblock Better 3-coloring algorithms: Excluding a triangle and a seven vertex
  path.
\newblock {\em Theoretical Computer Science}, 850:98--115, 2021.

\bibitem{Brause2022}
C.~Brause, M.~Gei{\ss}er, and I.~Schiermeyer.
\newblock Homogeneous sets, clique-separators, critical graphs, and optimal
  $\chi$-binding functions.
\newblock {\em Discrete Appl. Math.}, 320:211--222, 2022.

\bibitem{Bruce2009}
D.~Bruce, C.~T. Ho{\`a}ng, and J.~Sawada.
\newblock A certifying algorithm for 3-colorability of {$P_5$}-free graphs.
\newblock In Y.~Dong, D.-Z. Du, and O.~Ibarra, editors, {\em Algorithms and
  Computation}, pages 594--604, Berlin, Heidelberg, 2009. Springer Berlin
  Heidelberg.

\bibitem{CaiGoedgebeurHuang2023}
Q.~Cai, J.~Goedgebeur, and S.~Huang.
\newblock {Some results on $k$-critical $P_5$-free graphs}.
\newblock {\em Discrete Appl. Math.}, 334:91--100, 2023.

\bibitem{P5banner2019}
Q.~Cai, S.~Huang, T.~Li, and Y.~Shi.
\newblock Vertex-critical $({P_5},banner)$-free graphs.
\newblock In {\em Frontiers in Algorithmics - 13th International Workshop, FAW
  2019, Sanya, China, April 29 - May 3, 2019, Proceedings}, volume 11458, pages
  111--120. Springer, 2019.

\bibitem{CameronHoang2023P5C5}
B.~Cameron and C.~T. Hoàng.
\newblock Infinite families of $k$-vertex-critical $({P}_5, {C}_5)$-free
  graphs.
\newblock {\em Graphs Combin.}, 40:30, 2024.

\bibitem{CameronHoangSawada2022}
B.~Cameron, C.~T. Hoàng, and J.~Sawada.
\newblock {Dichotomizing $k$-vertex-critical $H$-free graphs for $H$ of order
  four}.
\newblock {\em Discrete Appl. Math.}, 312:106--115, 2022.
\newblock Ninth Workshop on Graph Classes, Optimization, and Width Parameters.

\bibitem{KCameron2021}
K.~Cameron, J.~Goedgebeur, S.~Huang, and Y.~Shi.
\newblock $k$-{C}ritical graphs in {$P_5$}-free graphs.
\newblock {\em Theoret. Comput. Sci.}, 864:80--91, 2021.

\bibitem{Chud4critical2020}
M.~Chudnovsky, J.~Goedgebeur, O.~Schaudt, and M.~Zhong.
\newblock {Obstructions for three-coloring and list three-coloring $H$-free
  graphs}.
\newblock {\em SIAM J. Discrete Math.}, 34(1):431--469, 2020.

\bibitem{Chudnovsky4criticalconnected2020}
M.~Chudnovsky, J.~Goedgebeur, O.~Schaudt, and M.~Zhong.
\newblock Obstructions for three-coloring graphs without induced paths on six
  vertices.
\newblock {\em J. Combin., Ser. B}, 140:45--83, 2020.

\bibitem{ChudnovskyHajebiSpirkl2024}
M.~Chudnovsky, S.~Hajebi, and S.~Spirkl.
\newblock {List-$k$-Coloring $H$-Free Graphs for All $k>4$}.
\newblock {\em Algorithmica}, (to appear), 2024.

\bibitem{ChudnovskyStacho2018}
M.~Chudnovsky and J.~Stacho.
\newblock {$3$-Colorable Subclasses of $P_8$-Free Graphs}.
\newblock {\em SIAM Journal on Discrete Mathematics}, 32(2):1111--1138, 2018.

\bibitem{CORNEIL1981163}
D.~Corneil, H.~Lerchs, and L.~Burlingham.
\newblock Complement reducible graphs.
\newblock {\em Discrete Applied Mathematics}, 3(3):163--174, 1981.

\bibitem{Couturier2015}
J.~F. Couturier, P.~A. Golovach, D.~Kratsch, and D.~Paulusma.
\newblock {List Coloring in the Absence of a Linear Forest}.
\newblock {\em Algorithmica}, 71(1):21--35, 2015.

\bibitem{Dhaliwal2017}
H.~S. Dhaliwal, A.~M. Hamel, C.~T. Ho{\`{a}}ng, F.~Maffray, T.~J.~D. McConnell,
  and S.~A. Panait.
\newblock {On color-critical $(P_5,$co-$P_5)$-free graphs}.
\newblock {\em Discrete Appl. Math.}, 216:142--148, 2017.

\bibitem{Erdos}
P.~Erd\H{o}s.
\newblock Graph theory and probability.
\newblock {\em Canadian J. Math.}, 11:34--38, 1959.

\bibitem{Esperetetal2013}
L.~Esperet, L.~Lemoine, F.~Maffray, and G.~Morel.
\newblock {The chromatic number of $\{P_5,K_4\}$-free graphs}.
\newblock {\em Discrete Mathematics}, 313(6):743--754, 2013.

\bibitem{HajebiLiSpirkl2022}
S.~Hajebi, Y.~Li, and S.~Spirkl.
\newblock {Complexity Dichotomy for List-5-Coloring with a Forbidden Induced
  Subgraph}.
\newblock {\em SIAM J. Discrete Math.}, 36(3):2004--2027, 2022.

\bibitem{Hoang2010}
C.~T. Ho{\`{a}}ng, M.~Kami{\'{n}}ski, V.~Lozin, J.~Sawada, and X.~Shu.
\newblock Deciding $k$-colorability of {$P_5$}-free graphs in polynomial time.
\newblock {\em Algorithmica}, 57:74--81, 2010.

\bibitem{Hoang2015}
C.~T. Ho{\`{a}}ng, B.~Moore, D.~Recoskie, J.~Sawada, and M.~Vatshelle.
\newblock {Constructions of $k$-critical $P_5$-free graphs}.
\newblock {\em Discrete Appl. Math.}, 182:91--98, 2015.

\bibitem{Huang2016}
S.~Huang.
\newblock {Improved complexity results on $k$-coloring $P_t$-free graphs}.
\newblock {\em European J. Combin.}, 51:336--346, 2016.

\bibitem{Huang2019}
S.~Huang, T.~Li, and Y.~Shi.
\newblock {Critical $(P_6,banner) $-free graphs}.
\newblock {\em Discrete Appl. Math.}, 258:143--151, 2019.

\bibitem{Jooken2026}
J.~Jooken.
\newblock {Vertex-critical $(P_5,\text{chair})$-free and
  $(P_5,\text{cricket})$-free graphs}.
\newblock {\em preprint, arXiv:2605.28537 [math.CO]}, 2026.

\bibitem{5critP6bullfree}
Y.~Ju, J.~Jooken, J.~Goedgebeur, and S.~Huang.
\newblock {There Are Finitely Many $5$-Vertex-Critical $( P_6 , bull)$-Free
  Graphs}.
\newblock {\em Journal of Graph Theory}, 112(3):255--266, 2026.

\bibitem{Kaminski2019}
M.~Kamiński and A.~Pstrucha.
\newblock Certifying coloring algorithms for graphs without long induced paths.
\newblock {\em Discrete Appl. Math.}, 261:258--267, 2019.

\bibitem{Karp1972}
R.~M. Karp.
\newblock {Reducibility among combinatorial problems}.
\newblock In {\em Complexity of computer computations ({P}roc. {S}ympos., {IBM}
  {T}homas {J}. {W}atson {R}es. {C}enter, {Y}orktown {H}eights, {N}.{Y}.,
  1972)}, pages 85--103, 1972.

\bibitem{MaffrayMorel2012}
F.~Maffray and G.~Morel.
\newblock On $3$-colorable {$P_5$}-free graphs.
\newblock {\em SIAM J. Discrete Math.}, 26(4):1682--1708, 2012.

\bibitem{McConnell2011}
R.~McConnell, K.~Mehlhorn, S.~Näher, and P.~Schweitzer.
\newblock Certifying algorithms.
\newblock {\em Comput. Sci. Rev.}, 5(2):119--161, 2011.

\bibitem{mirsky1971dual}
L.~Mirsky.
\newblock A dual of {D}ilworth's theorem.
\newblock {\em The American Mathematical Monthly}, 78(8):876--877, 1971.

\bibitem{Ramsey}
F.~P. Ramsey.
\newblock {On a problem of formal logic}.
\newblock {\em Proc. Lond. Math. Soc.}, (s2-30):264--286, 1928.

\bibitem{RanderathSchiermeyer2010}
B.~Randerath and I.~Schiermeyer.
\newblock {$3$-Colorability $\in$P for $P_6$-free graphs}.
\newblock {\em Discrete Applied Mathematics}, 136(2):299--313, 2004.
\newblock The 1st Cologne-Twente Workshop on Graphs and Combinatorial
  Optimization.

\bibitem{WestGraphTheoryBook}
D.~B. West.
\newblock {\em {Introduction to Graph Theory}}.
\newblock Prentice Hall, Inc., Upper Saddle River, NJ, 1996.

\bibitem{P5W4Journal}
W.~Xia, J.~Jooken, J.~Goedgebeur, I.~Beaton, B.~Cameron, and S.~Huang.
\newblock {Vertex-critical $(P_5, W_4)$-free graphs}.
\newblock {\em Theoretical Computer Science}, 1075:115934, 2026.

\bibitem{Xiaetal2023}
W.~Xia, J.~Jooken, J.~Goedgebeur, and S.~Huang.
\newblock {Critical $(P_5,dart)$-Free Graphs}.
\newblock In {\em Combinatorial Optimization and Applications: 16th
  International Conference, COCOA 2023, Hawaii, HI, USA, December 15–17,
  2023, Proceedings, Part II}, page 390–402, Berlin, Heidelberg, 2023.
  Springer-Verlag.

\bibitem{xia2024results}
W.~Xia, J.~Jooken, J.~Goedgebeur, and S.~Huang.
\newblock {Some Results on Critical $(P_5,H)$-free Graphs}.
\newblock {\em Theoretical Computer Science}, 1051:115411, 2025.

\bibitem{zhouetal2025P11C3freecritical}
Y.~Zhou, J.~Jooken, B.~Shan, J.~Goedgebeur, and S.~Huang.
\newblock {Three-coloring triangle-free graphs without long forbidden paths}.
\newblock {\em preprint, arXiv:2512.12349 [math.CO]}, 2025.

\end{thebibliography}
\end{document}